\documentclass[12pt,draft]{amsart}
\usepackage{setspace,units}

\newtheorem{theorem}{Theorem}[section]

\newtheorem{Corollary}[theorem]{Corollary}

\theoremstyle{definition}
\newtheorem{definition}[theorem]{Definition}

\newtheorem{remark}[theorem]{Remark}

\RequirePackage{amsmath}
\RequirePackage{amssymb}
\RequirePackage{amsthm}
 \RequirePackage{epsfig}

\begin{document}

\date{}
\title[Algebraic methods in sum-product phenomena  ]
{Algebraic methods in sum-product phenomena}

\author{ Chun-Yen Shen}

\address{Department of Mathematics, Indiana University, Bloomington, IN 47405, USA}
\email{shenc@umail.iu.edu}



\date{}

\begin{abstract}
We classify the polynomials $f(x,y) \in \mathbb R[x,y]$ such that given any finite set $A \subset \mathbb R$ if $|A+A|$ is small, then $|f(A,A)|$ is large. In particular, the following bound holds : $|A+A||f(A,A)| \gtrsim |A|^{5/2}.$ The Bezout's theorem and a theorem by Y. Stein play important roles in our proof.
\end{abstract}

\maketitle


\section{introduction}

The sum-product problems have been intensively studied since the
work by Erd\H os and Szemer\'edi \cite{ESZ83} that there exists $c > 0$ such that for any finite set $A \subset \mathbb Z$, one has
$$ \max (|A+A|, |A\cdot A|) \gtrsim{|A|}^{1+c}.$$ Later, much work has been done either to give an explicit bound of $c$ or to give a generalization of the sum-product theorem. One of the important generalizations is the work by Elekes, Nathanson and Ruzsa \cite{ENR00} who showed that given any finite set $A \subset \mathbb R$, let f be a strictly convex ( or concave )
function defined on an interval containing $A.$ Then $$ \max (|A+A|, |f(A)+ f(A)|) \gtrsim{|A|}^{5/4}.$$ Taking $f(x)=\log x$ recovers the sum-product theorem mentioned above by Erd\H os and Szemer\'edi. An analogous result in finite field $\mathbb F_p$ with $p$ prime was proven in 2004 by Bourgain, Katz and Tao \cite{BKT04} that
if  $p^{\delta} < |A| < p^{1-\delta}$, for
some $\delta>0$, then there exists $\epsilon=\epsilon(\delta)>0$
such that
$$ \max (|A+A|, |A\cdot A|) \gtrsim{|A|}^{1+\epsilon}.$$ This remarkable result
has found many important applications in various areas ( see \cite{Bo05}, \cite{Bo09} for further discussions ).
Recently, Solymosi (\cite{So082}) applied spectral graph theory to give a
similar result mentioned above by Elekes, Nathanson and Ruzsa showing that for a class of functions $f$, one has the following bound.
$$\max(|A+B|,|f(A)+C|)\gtrsim \min(|A|^{1/2}q^{1/2},|A||B|^{1/2}|C|^{1/2}q^{-1/2}),$$ for any $A,B,C \subset \mathbb F_q.$ This was further studied by Hart, Li and the author \cite{HLS09} using Fourier analytic methods showing that for suitable assumptions on the functions $f$ and $g$, one has the bound.
 $$\max(|f(A)+B|,|g(A)+C|) \gtrsim \min(|A|^{1/2}q^{1/2},|A||B|^{1/2}|C|^{1/2}q^{-1/2}).$$
A natural and important question one may ask is to classify
two variables polynomials $f(x,y) \in \mathbb F[x,y]$ such
that when $|A+A|$ is small, then $|f(A,A)|$ is large. This problem was first raised and studied by Vu in \cite{Vu08}. Considering $A$ an arithmetics progression shows that if $f(x,y)$ is linear then $|A+A|$ and $|f(A,A)|$ can be small at the same time. More generally, if $f(x,y)=Q(L(x,y)),$ where $L(x,y)$ is linear and $Q$ is a one variable polynomial, then again $|A+A|$ and $|f(A,A)|$ can be small at the same time. This consideration reveals that if $f(x,y)$ is not like $Q(L)$, we should have $|f(A,A)|$ is large when $|A+A|$ is small. Indeed, this was confirmed by Vu \cite{Vu08} using spectral graph theory
showing that if $f(x,y)$ can not be written as $Q(L(x,y)),$ then one has the following bound.
 $$\max(|A+A|,|f(A,A)|)\gtrsim
\min(|A|^{2/3}q^{1/3},|A|^{3/2}q^{-1/4}),$$
for any $A \subset \mathbb F_q.$ This was also the first time using spectral graph theory to study the incidence problems ( see \cite{HLS09}, in which Fourier analytic methods were given to reprove the results by Vu). However this result is only effective when $|A| \geq q^{1/2}.$ Therefore it turns out that if one wants to extend this result to the real setting, new tools are required. As observed by Elekes \cite{El97}, the sum-product problems have interesting connections to the problems in incidence geometry. In particular, he applied the so-called Szemer\'edi-Trotter theorem to show that one can take $c=1/4$ in the above Erd\H os-Szemer\'edi's sum-product theorem. Indeed, in this paper we apply a generalization of Szemer\'edi-Trotter theorem by Sz\'ekely \cite{LS} to establish an analogous result in the reals. Namely, given non-degenerate polynomial $f(x,y) \in \mathbb R[x,y]$ (see section 2 for the definition), then for any finite set $A \subset \mathbb R$ one has the following bound.
$$\max(|A+A|,|f(A,A)|) \gtrsim |A|^{5/4}.$$
One may find the difficulties come from the reducibilities of the polynomials $f(x,y),$ and this is how the Bezout's theorem and a theorem by Y. Stein concerning the reducibility of a multi-variables polynomial come into our proof.

\section{Algebraic Preliminaries}
Given quantities $X$ and $Y$ we use the notation
$X \lesssim Y$ to mean $X \leq C Y,$
where the constant $C$ is universal (i.e. independent of
$A$). The constant $C$ may vary from line to line but are
universal. It is also clear that when one of the quantities $X$ and $Y$ has polynomial $f(x,y)$ involved, the constant $C$ may also depend on the degree of $f$. We now state some definitions and give some preliminary lemmas. The first two definitions can be found in \cite{Vu08} and \cite{YS} respectively. For the convenience of the reader, we state the definitions here.

\begin{definition} A polynomial $f(x,y) \in \mathbb R[x,y]$ is
degenerate if it can be written as $Q(L(x,y))$ where $Q$ is a
one-variable polynomial and $L$ is a linear form in $x,y$.
\end{definition}

\begin{definition}
A polynomial $f(x,y) \in \mathbb C[x,y]$ is composite if it can be written as $Q(g(x,y))$ for some $g(x,y) \in \mathbb C[x,y],$ and some $Q(t) \in \mathbb C[t]$ of degree $\geq 2.$ \end{definition}

\begin{definition}
Given a polynomial $f(x,y) \in \mathbb R[x,y],$ we use $\deg_x(f)$ to denote the degree of $f$ in $x$ variable ( i.e. consider $y$ as a constant). Similarly, denote $\deg_y(f)$ the degree of $f$ in $y$ variable.
\end{definition}

The following theorem is the celebrated Bezout theorem, and the next one is a theorem by Y. Stein [9].

\begin{theorem}(Bezout's theorem)
Two algebraic curves of degree $m$ and $n$ intersect
in at most $mn$ points unless they have a common factor.
\end{theorem}

\begin{theorem}(Y. Stein)
Given $f(x,y)\in \mathbb C[x,y]$ of degree $k$. Let $\sigma(f)=\{\lambda : f(x,y)-\lambda  \text{ is reducible}\}.$ Suppose $f(x,y)$ is not composite, then $|\sigma(f)| < k.$
\end{theorem}

We shall need  a theorem by
Sz\'ekely \cite{LS}, which is a generalization
of Szemer\'edi-Trotter incidence theorem in the plane.

\begin{theorem}
Let $P$ be a finite collection of points in $\mathbb R^2$, and $L$
be a finite collection of curves in $\mathbb R^2$. Suppose that
for any two curves in $L$ intersect in at most $\alpha$ points,
and any two points in $P$ are simultaneously incident to at most
$\beta$ curves. Then
$$I(P,L)=|\{(p.\ell) \in P \times L : p \in \ell\}| \leq (\alpha^{1/2}\beta^{1/3}|P|^{2/3}|L|^{2/3}+|L|+\beta|P|).$$
\end{theorem}

\section{main results}
As discussed in section 1, we will be applying the Sz\'ekely's theorem. Therefore we need to take the advantage of the non-degeneracy property of the polynomial to construct a bunch of curves which each of them has large intersections with some appropriate points set $P$. In order to apply the Sz\'ekely's theorem efficiently, we need to control the number of the curves. It turns out that we shall need the following theorems.

\begin{theorem} Given $f(x,y) \in \mathbb R[x,y]$ of degree $k \geq 2,$ and assume that $\deg_x(f) \geq \deg_y(f)$. Suppose there exists distinct $a_1,..,a_{k^2+1},$ and $b_1,..,b_{k^2+1}$ and a
polynomial
$$Q(t)=q_mt^m+q_{m-1}t^{m-1}+...+q_0$$ so that
$$f(x,a_i)=Q(x+b_i)$$ for each $i$. Then $f(x,y)=Q(g(x,y))$ for some $g(x,y),$ and $\deg Q \geq 2.$
\end{theorem}

\begin{proof}
First we write $f(x,y)=c_kx^k+\cdots+x^m(a_{k-m}'y^{k-m}+a_{k-m-1}'y^{k-m-1}+\cdots+a_0')+x^{m-1}(a_{k-m+1}y^{k-m+1}+a_{k-m}y^{k-m}+\cdots+a_0)+ h(x,y)$, where
$h(x,y)$ is the lower degree terms in $x$ of $f(x,y).$ By assumptions, for each $i$ we have
$$f(x,a_i)=c_kx^k+\cdots+x^m(\sum_{h=0}^{k-m}a_h'a_{i}^h)+x^{m-1}(\sum_{h=0}^{k-m+1}a_ha_{i}^h)+h(x,a_i)$$
which is equal to
$$Q(x+b_i)=q_m(x+b_i)^m+q_{m-1}(x+b_i)^{m-1}+\cdots.$$
We compare the coefficients of the term $x^m$. By our assumption on  $a_i,$ we first conclude that $f(x,y)$ doesn't have $x^l$ terms for $l > m$, and $a_h'=0$ for $h=1 \sim k-m,$ and $a_0'=q_m.$
We compare the coefficients of the term $x^{m-1}$ to get
$$q_mmb_{i}+q_{m-1}=\sum_{h=0}^{k-m+1}a_ha_{i}^{h},$$
which gives  $b_i=\frac{\sum_{h=0}^{k-m+1}a_ha_{i}^{h}-q_{m-1}}{q_mm}$,
for each $1\leq i \leq k^2+1$. Now given any $x_0 \in \mathbb R$,
$f(x_0,y)-Q(x_0+\frac{\sum_{h=0}^{k-m+1}a_hy^h-q_{m-1}}{q_mm})$ is a polynomial in $y$ of degree
$\leq \max\{k,m(k-m+1)\} \leq k^2$, but is zero for distinct $k^2+1$ values of $a_i$. Therefore we
conclude that $f(x,y)=Q(x+\frac{\sum_{h=0}^{k-m+1}a_hy^{h}-q_{m-1}}{q_mm}).$ Since we assume $\deg(f) \geq 2,$ we also conclude that $\deg Q \geq 2,$ otherwise it will contradict the assumption that $\deg_x(f) \geq \deg_y(f).$
\end{proof}

\begin{Corollary}
 Given $f(x,y) \in \mathbb R[x,y]$
of degree $k \geq 2$. Suppose $\deg_x(f) \geq \deg_y(f),$ and there exists $(k^3+1)$ distinct points
$S=\{(a_i,b_i)\}_{i=1}^{k^3+1}$ in the plane such that
\[f(x-a_1,b_1)=f(x-a_2,b_2)=\cdots=f(x-a_{k^3+1},b_{k^3+1}).\]
Then $f$ is composite.
\end{Corollary}

\begin{proof}
 Suppose there are $\geq k^2+1$ distinct $b_i$ such that $f(x-a_1,b_1)=f(x-a_2,b_2)=\cdots=f(x-a_i,b_i),$ we then apply theorem 3.1 to get that $f$ is composite. If not, there must exist one $b \in \{b_i :  (a_i,b_i) \in S\}$ such that there are $\geq k+1$ distinct $a_i$ so that $f(x-a_1,b)=f(x-a_2,b)=\cdots=f(x-a_{k+1},b).$ A direction computation shows that in this case the only possible is that $f$ is a one variable polynomial in $y$ of degree $\geq 2,$ which is composite.
\end{proof}

\begin{remark}
The assumption $\deg_x(f) \geq \deg_y(f)$ is necessary because we might have the case $f(x,y)=x+y^2.$
\end{remark}

\begin{theorem}
Given non degenerate polynomial $f(x,y)$ of degree $k$. Then for any finite set $A \subset \mathbb R$, one has
$$|A+A||f(A,A)| \gtrsim |A|^{5/2}.$$
\end{theorem}

Before we proceed to prove our main theorem, we observe that our non-degenerate polynomial $f(x,y)$ could be $Q(g(x,y))$ for some $Q(t) \in \mathbb R[t]$ and some non-degenerate polynomial $g(x,y)$. In this case, we will work on $g(x,y)$ instead of $f(x,y)$, since we are concerned the cardinality and we use a fact that $|f(A,A)| \geq \frac{1}{deg Q} |g(A,A)|,$ which in turn says that we can assume $f(x,y)$ is not composite. In addition, we can always assume the $\deg_x(f) \geq \deg_y(f)$, since again we are concerned $|f(A,A)|$ ( for example if $f(x,y)=x+y^2$, we write it as $x^2+y$).

\begin{proof}

Given $y_0 \in \mathbb R$, let $f_{y_0}(x)=f(x,y_0).$ We first remove the elements $b$ in $A$ such that
$f(x,b)$ is identically zero. Since $f$ is of degree $k$, there are at most $k$ elements $b$ which make this happen. We now abuse the notation, let $A=A-\{b_1,..,b_k\},$ where $f(x,b_i)$ is identically zero. Now given
$(a,b) \in A \times A,$  let $l_{a,b}=\{(x+a, f_{b}(x)) : x \in
\mathbb R \}$, and let $L=\{l_{a,b}: a, b \in A \}$. Furthermore
for each $(a,b) \in A \times A$, we represent $l_{a,b}$ by
$\{(x+a, T_{b}(x+a) : x \in \mathbb R\}= \{(x',T_{(a,b)}(x')) : x'
\in \mathbb R\}$, where $T_b(x+a)$ is the Taylor polynomial of
$f_b(x)$ about $a$. Let us write $T_b(x+a)=$$\sum_{j=0}^k c_j(x+a)^j$ for some
$c_j$ and $T_{(a,b)}(x)=\sum_{j=0}^k c_jx^j$. We note that for each
$(a,b)$, $T_{(a,b)}(x)$ is a polynomial of degree $\leq k.$ Therefore for any
two pairs $(a,b),(c,d) \in A \times A$, if $T_{(a,b)}(x)$ intersects
$T_{(c,d)}(x)$ more than $k+1$ points, then $T_{(a,b)}(x)=T_{(c,d)}(x)$.
Now given $(a,b),(c,d) \in A \times A$, we say $(a,b) \sim (c,d)$
if and only if $T_{(a,b)}(x)=T_{(c,d)}(x).$ This is equivalently saying
the Taylor polynomials of $f_b(x)$ and $f_d(x)$ about $a$ and $ c$
respectively have the same form, i.e.
$$T_{b}(x+a)=c_k(x+a)^k+c_{k-1}(x+a)^{k-1}+\cdots+c_0$$ and
$$T_{d}(x+c)=c_k(x+c)^k+c_{k-1}(x+c)^{k-1}+\cdots+c_0.$$
First we observe that $T_{(a,b)}(x)=f(x-a,b),$ we now apply Corollary 3.2 to get that each equivalence class has
at most $k^3$ elements. Therefore $L'=\{T_{(a,b)}(x): a,b \in A\}/ \sim$
has at least $\frac{|A|^2}{k^3}$ equivalence classes. For each
equivalence class we choose one represented curve, and conclude
that there are $\geq \frac{|A|^2}{k^3}$ curves, and any two of them
intersect at most $k$ points. We now show that for most pairs of points in $P=(A+A)\times f(A,A),$ there are at most $k^2$ curves from $L'$ which are
incident to them simultaneously.  We note that if the curve $T_{(a,b)}$ incident to some point $p=(x',y') \in P,$ we have $f(x'-a,b)=y'.$  Given any two points $p_1=(x_0,y_0), p_2=(x_0',y_0') $ in $P=(A+A) \times f(A,A).$ Consider
two algebraic curves $f(x_0-x,y)-y_0=0$ and $f(x_0'-x,y)-y_0'=0.$ If there is a curve incident to $p_1$ and $p_2$ simultaneously, then there exists a pair $(a,b)$ such that these two algebraic curves intersect at $(a,b)$. By Bezout's theorem,
there are at most $k^2$ pairs $(a,b)$ so that $$f(x_0-a,b)-y_0=0$$ and $$f(x_0'-a,b)-y_0'=0,$$ unless these two algebraic curves $f(x_0-x,y)-y_0=0$ and $f(x_0'-x,y)-y_0'=0$ have a common factor, which means
$$f(x_0-x,y)-y_0=G(x,y)H(x,y),$$
and
$$f(x_0'-x,y)-y_0'=G(x,y)H'(x,y).$$
This implies $$f(x,y)-y_0=G(x_0-x,y)H(x_0-x,y),$$ and $$f(x,y)-y_0'=G(x_0'-x,y)H'(x_0'-x,y),$$
 which in turn shows that the $y$ coordinates of the points $p_1$ and $p_2$ are from $\sigma(f).$ Therefore by Stein's theorem, we conclude that we can remove at most $|A+A|k$ points from $P=(A+A)\times f(A,A)$, and any pair in the rest of points in $P$ has at most $k^2$ curves incident to them simultaneously. Therefore we let $P'=P - \{(A+A)\times \sigma(f)\},$ and observe that each curve $T_{(a,b)} \in L'$ incidents to at least $|A|/k$ points in $P'$. We now apply theorem 2.5 on $P'$ and $L'$ to get
$$\frac{|A|^2}{k^3}\frac{|A|}{k} \lesssim (|P'|\frac{|A|^2}{k^3})^{2/3},$$ which implies $|A+A||f(A,A)| \gtrsim |A|^{5/2}.$
\end{proof}

{\bf Acknowledgment :} The author would like to thank his colleague Tuyen Truong and Liangpan Li for some helpful comments.

\end{document}